\documentclass[12pt]{amsart}
\usepackage[alphabetic]{amsrefs}
\usepackage{graphicx}
\usepackage{amscd}
\usepackage{upgreek}
\usepackage{stmaryrd}
\usepackage{longtable}
\usepackage[T1]{fontenc}
\usepackage{latexsym, amsmath, amssymb, amsthm}
\usepackage{rsfso}
\usepackage[mathscr]{eucal}
\usepackage{mathtools}
\usepackage{mathptmx}
\usepackage{titletoc}
\usepackage{wrapfig}
\usepackage{float}
\usepackage{xypic}
\usepackage{microtype}
\usepackage{dsfont}
\usepackage{xcolor}
\usepackage{color}
\usepackage[colorlinks,linkcolor=blue,anchorcolor=blue,urlcolor=blue,
citecolor=blue]{hyperref}
\usepackage{hyperref}
\allowdisplaybreaks	
\usepackage[centering, includeheadfoot, hmargin=1.0in, tmargin=0.5in, 
  bmargin=0.6in, headheight=6pt]{geometry}
\newtheorem{theorem}{Theorem}[section]

\newtheorem{lem}[theorem]{Lemma}

\newtheorem{thm}[theorem]{Theorem}

\DeclareMathOperator{\GL}{GL}

\DeclareMathOperator{\SL}{SL}

\newcommand{\N}{\mathbb{N}} 
\newcommand{\F}{\mathbb{F}} 

\newcommand{\ra}{\longrightarrow} 
\newcommand{\A}{\mathcal{A}} 
\newcommand{\B}{\mathcal{B}} 
\newcommand{\LL}{\mathcal{L}} 
\newcommand{\C}{\mathcal{C}} 
\newcommand{\NN}{\mathcal{N}} 
 
\newcommand{\I}{\mathcal{I}}

\begin{document}
\title{Relations between modular invariants of a vector and a covector in dimension two} 
\def\shorttitle{Relations between modular invariants of a vector and a covector in dimension two}

\author{Yin Chen}
\address{School of Mathematics and Statistics, Northeast Normal University, Changchun 130024, China \& Department of Mathematics and Statistics, Queen's University, Kingston, K7L 3N6, Canada}
\email{ychen@nenu.edu.cn}

\begin{abstract}
We exhibit a set of generating relations for the modular invariant ring of a vector and a covector for the two-dimensional  general linear group over a finite field.
\end{abstract}

\date{\today}
\maketitle \baselineskip=15pt

\dottedcontents{section}[1.16cm]{}{1.8em}{5pt}
\dottedcontents{subsection}[2.00cm]{}{2.7em}{5pt}

\vspace{-4mm}
\section{Introduction}
\setcounter{equation}{0}
\renewcommand{\theequation}
{1.\arabic{equation}}
\setcounter{theorem}{0}
\renewcommand{\thetheorem}
{1.\arabic{theorem}}

\noindent Let $\F_q$ be a finite field of order $q=p^s (s\in\N^+)$ and $V$ be an $n$-dimensional vector space over $\F_q$.
Let $\GL(V)$ be the general linear group on $V$ and $V^*$ be the dual space of $V$.  Suppose $G\leqslant \GL(V)$ is a subgroup. The diagonal action of $G$ on $V\oplus V^*$ algebraically extends  to a degree-preserving $\F_q$-automorphism action on the symmetric algebra $\F_q[V\oplus V^*]$. Choosing a basis $\{y_1,\dots,y_n\}$ of $V$ and a basis $\{x_1,\dots,x_n\}$ for $V^*$ dual to $\{y_1,\dots.y_n\}$, we may identify $\GL(V)$ with $\GL_n(\F_q)$ and identify the symmetric algebra $\F_q[V\oplus V^*]$ with the polynomial algebra $\F_q[x_1,\dots,x_n,y_1,\dots,y_n]$ respectively. The invariant ring $\F_q[V\oplus V^*]^{G}$
consisting of all polynomials in $\F_q[V\oplus V^*]$ fixed by every element of $G$ is the main object of study in the invariant theory of a vector and a covector.

Giving a presentation via generators and relations for an invariant ring $\F_q[V\oplus V^*]^{G}$ is a fundamental and difficult task. 
When $G=U_n(\F_q)$, the unipotent group of upper triangular matrices with 1’s on the diagonal, 
\cite[Theorem 2.4]{BK11} has proved that $\F_q[V\oplus V^*]^{U_n(\F_q)}$ is a complete intersection 
via presenting a minimal generating set for the invariant ring and generating relations among these generators. 
A minimal generating set for $\F_q[V\oplus V^*]^{\GL(V)}$ was conjectured by \cite[Conjecture 3.1]{BK11} and recently
was confirmed by \cite[Theorem 1]{CW17}.
In \cite[Section 12]{CW17}, we also give a conjectural set of generating relations between
these generators of $\F_q[V\oplus V^*]^{\GL(V)}$.

The purpose of this short note is to confirm Conjecture 16 appeared in \cite{CW17} for 
the special case $n=2$ where there are seven generators for the invariant ring and it was originally conjectured that 
five generating relations appear. Note that the generating relations  in \cite[Conjecture 16]{CW17} were not all given explicitly, even in the case $n=2$.

In Section \ref{sec2}, we recall the minimal generating set for $\F_q[V\oplus V^*]^{\GL_2(\F_q)}$ and present the three generating relations that appeared in \cite{CW17} previously. Section \ref{sec3} contains constructions of two new relations for the invariant ring. In Section \ref{sec4}, we show that any relations for $\F_q[V\oplus V^*]^{\GL_2(\F_q)}$ can be generated by  these five relations; see Theorem \ref{stm}. Section \ref{sec5} consists of three lemmas that will be 
necessary to the proof of Theorem \ref{stm}.

\subsection*{Acknowledgements} The author thank the anonymous referee for a careful reading of the manuscript and for  constructive corrections and suggestions.





\section{$\F_q[V\oplus V^*]^{\GL_2(\F_q)}$}\label{sec2}
\setcounter{equation}{0}
\renewcommand{\theequation}
{2.\arabic{equation}}
\setcounter{theorem}{0}
\renewcommand{\thetheorem}
{2.\arabic{theorem}}

\noindent In this preliminary  section, we recall the minimal generating set for $\F_q[V\oplus V^*]^{\GL_2(\F_q)}$ and some relations appeared in \cite{CW17}.
Let $\F_q[V\oplus V^*]=\F_q[x_1,x_2,y_1,y_2]$ be the polynomial ring with an algebraic involution $\ast$ given by $x_1\mapsto y_2,x_2\mapsto y_1,y_1\mapsto x_2$ and $y_2\mapsto x_1$. We regard $\F_q[V]=\F_q[x_1,x_2]$ and $\F_q[V^*]=\F_q[y_1,y_2]$ as subalgebras of $\F_q[V\oplus V^*]$. We define
$$d_0:=\det\begin{pmatrix}
     x_1^q &x_2^q    \\
     x_1^{q^2} &x_2^{q^2}
\end{pmatrix},d_1:=\det\begin{pmatrix}
     x_1 &x_2   \\
     x_1^{q^2} &x_2^{q^2}
\end{pmatrix},d_2:=\det\begin{pmatrix}
     x_1 &x_2   \\
     x_1^{q} &x_2^{q}
\end{pmatrix}$$
and $c_0:=d_0/d_2=d_2^{q-1},c_1:=d_1/d_2$. Denote by $f^*$ the image of an element $f\in \F_q[V\oplus V^*]$ under the involution $*$.
Then $\F_q[V]^{\GL_2(\F_q)}=\F_q[c_0,c_1]$ and $\F_q[V^*]^{\GL_2(\F_q)}=\F_q[c_0^*,c_1^*]$. We define 
\begin{eqnarray*}
u_{-i}:= x_{1}y_{1}^{q^i}+x_{2}y_{2}^{q^i},&
u_{0}:= x_{1}y_{1}+x_{2}y_{2},&
u_{i}:= x_{1}^{q^i}y_{1}+x_{2}^{q^i}y_{2}
\end{eqnarray*}
for $i\in \N^+$.
We have seen that $\F_{q}[V\oplus V^{*}]^{\GL_{2}(\F_{q})}$ is generated minimally by $\{c_{0},c_{1},c_{0}^{*},c_{1}^{*},u_{-1},u_{0},u_{1}\}$; see \cite[Theorem 1]{CW17}.

To find all relations among the above seven  generators of $\F_{q}[V\oplus V^{*}]^{\GL_{2}(\F_{q})}$, we start with the following initial relation:
\begin{equation}\tag{$T_0$}
\label{T0}
c_{0}\cdot u_{0}-c_{1}\cdot u_{1}+u_{2} =0.
\end{equation}
The Frobenious map $F^*:\F_{q}[V\oplus V^{*}]\ra \F_{q}[V\oplus V^{*}]$ defined by $x_i\mapsto x_i$ and $y_i\mapsto y_i^q$ applies to (\ref{T0}) and we obtain the first relation:
\begin{equation}\tag{$T_1$}
\label{T1}
c_{0}\cdot u_{-1}-c_{1}\cdot u_{0}^{q}+u_{1}^{q} =0.
\end{equation}
Note that $u_{-i}^*=u_i$ for $i\in\N$. Applying the involution $*$ to (\ref{T1}) obtains the second one:
\begin{equation}\tag{$T_1^*$}
\label{TT1}
c_{0}^*\cdot u_{1}-c_{1}^*\cdot u_{0}^{q}+u_{-1}^{q} =0.
\end{equation}
Note that $d_2$ and $d_2^*$ are not $\GL_2(\F_q)$-invariants, but their product $d_2\cdot d_2^*$  belongs to
$\F_{q}[V\oplus V^{*}]^{\GL_{2}(\F_{q})}$. In fact, a direct computation verifies that
\begin{equation}\tag{$K_{00}$}
\label{K00}
d_{2}\cdot d_{2}^*=u_{-1}\cdot u_{1}-u_{0}^{q+1},
\end{equation}
and from which, we obtain  the third relation:
\begin{equation}\tag{$T_{00}$}
\label{T00}
c_{0}\cdot c_{0}^{*}-(u_{-1}\cdot u_{1}-u_{0}^{q+1})^{q-1}=0.
\end{equation}
These relations (\ref{T0}), (\ref{T1}), (\ref{TT1}), (\ref{K00}), and (\ref{T00}) are taken from \cite[Section 6]{CW17}.

\section{$\SL_2(\F_q)$-Invariants and More Relations} \label{sec3}
\setcounter{equation}{0}
\renewcommand{\theequation}
{3.\arabic{equation}}
\setcounter{theorem}{0}
\renewcommand{\thetheorem}
{3.\arabic{theorem}}

\noindent This section reveals two new relations by introducing  a couple of  $\SL_2(\F_q)$-invariants. For $0\leqslant s\leqslant q-1$, we define 
\begin{equation}
\label{ }
h_{s}:=\frac{u_{1}^{s+1}\cdot(d_{2}^{*})^{q-1-s}+u_{-1}^{q-s}\cdot d_{2}^{s}}{u_{0}^{q}}.
\end{equation}
We observe that $h_s^*=h_{q-1-s}$. In particular, $h_{0}=c_{1}^{*},h_{q-1}=c_{1}$ and each $h_{s}\in \F_{q}[V\oplus V^{*}]^{\SL_{2}(\F_{q})}$; see \cite[Lemma 2.6]{Che14}.
We also have seen in \cite[Lemma 2.7]{Che14} that 
\begin{equation}\tag{$R_s$}
\label{Rs}
h_s\cdot u_1=u_0\cdot u_{-1}^{q-1-s}\cdot d_2^s+d_2^*\cdot h_{s+1}
\end{equation}
for $0\leqslant s\leqslant q-2$. Moreover,

\begin{lem}
For $1\leqslant s\leqslant q-1$, we have 
\begin{equation}\tag{$K_s$}
\label{Ks}
h_{s}\cdot d_{2}^{*s}=c_{1}^{*}\cdot u_{1}^{s}+u_{-1}^{q-s}\cdot u_{0}\cdot\sum_{i=1}^s (-1)^i{s\choose i} (u_{-1}\cdot u_{1})^{s-i}(u_{0}^{q+1})^{i-1}.
\end{equation}
Applying the involution $*$ on (\ref{Ks}), we have
\begin{equation}\tag{$K_s^*$}
\label{KKs}
h_{q-1-s}\cdot d_{2}^{s}=c_{1}\cdot u_{-1}^{s}+u_{0}\cdot u_{1}^{q-s}\cdot\sum_{i=1}^s (-1)^i{s\choose i} (u_{-1}\cdot u_{1})^{s-i}(u_{0}^{q+1})^{i-1}.
\end{equation}
\end{lem}

\begin{proof}
Note that $u_0^q\cdot h_{s}\cdot d_{2}^{*s}=c_{0}^{*}\cdot u_{1}^{s+1}+u_{-1}^{q-s}\cdot(d_{2}\cdot d_{2}^{*})^{s}=
c_{0}^{*}\cdot u_{1}^{s+1}+u_{-1}^{q-s}\cdot(u_{-1}\cdot u_{1}-u_{0}^{q+1})^{s}$. It follows that (\ref{TT1}) that
$c_{0}^{*}\cdot u_{1}^{s+1}=c_{1}^*\cdot u_{0}^{q}\cdot u_1^s-u_{-1}^{q}\cdot u_1^s$. By the binomial formula, we see that
\begin{eqnarray*}
u_{-1}^{q-s}\cdot(u_{-1}\cdot u_{1}-u_{0}^{q+1})^{s}&=&u_{-1}^{q-s}\cdot\sum_{i=0}^s (-1)^i{s\choose i} (u_{-1}\cdot u_{1})^{s-i}(u_{0}^{q+1})^i \\
&=&u_{-1}^q\cdot u_1^s+u_{-1}^{q-s}\cdot\sum_{i=1}^s (-1)^i{s\choose i} (u_{-1}\cdot u_{1})^{s-i}(u_{0}^{q+1})^i\\
&=&u_{-1}^q\cdot u_1^s+u_{-1}^{q-s}\cdot u_0^{q+1}\cdot\sum_{i=1}^s (-1)^i{s\choose i} (u_{-1}\cdot u_{1})^{s-i}(u_{0}^{q+1})^{i-1}.
\end{eqnarray*}
Hence, $u_0^q\cdot h_{s}\cdot d_{2}^{*s}=c_{1}^*\cdot u_{0}^{q}\cdot u_1^s+u_{-1}^{q-s}\cdot u_0^{q+1}\cdot\sum_{i=1}^s (-1)^i{s\choose i} (u_{-1}\cdot u_{1})^{s-i}(u_{0}^{q+1})^{i-1}$. The proof will be completed via dividing with $u_0^q$ on the both sides of this equality.
\end{proof}

In particular, (\ref{Ks}) with $s=q-1$ implies the fourth relation:
\begin{equation}
\tag{$T_{10}$}
\label{T10}
c_{1}\cdot c_{0}^{*}-c_{1}^{*}\cdot u_{1}^{q-1}-u_{-1}\cdot u_{0}\cdot\sum_{i=1}^{q-1} (-1)^i{q-1\choose i} (u_{-1}\cdot u_{1})^{q-1-i}\cdot u_{0}^{(q+1)(i-1)}=0.
\end{equation}
Applying the involution $\ast$ to (\ref{T10}) (or taking $s=q-1$ in (\ref{KKs})), we obtain the fifth relation:
\begin{equation}
\tag{$T_{01}$}
\label{T01}
c_{0}\cdot c_{1}^{*}-c_{1}\cdot u_{-1}^{q-1}-u_{0}\cdot u_{1}\cdot\sum_{i=1}^{q-1} (-1)^i{q-1\choose i} (u_{-1}\cdot u_{1})^{q-1-i}\cdot u_{0}^{(q+1)(i-1)}=0.
\end{equation}

\section{The Main Theorem} \label{sec4}
\setcounter{equation}{0}
\renewcommand{\theequation}
{4.\arabic{equation}}
\setcounter{theorem}{0}
\renewcommand{\thetheorem}
{4.\arabic{theorem}}

\noindent In this section, we show that the five relations (\ref{T1}), (\ref{TT1}), (\ref{T00}), (\ref{T01}), and (\ref{T10})
constructed in previous two sections can generate any relation for $\F_{q}[V\oplus V^{*}]^{\GL_{2}(\F_{q})}$, confirming the special case $n=2$ in \cite[Conjecture 16]{CW17}.
Let $S:=\F_q[C_{0},C_{1},C_{0}^{*},C_{1}^{*},U_{-1},U_{0},U_{1}]$ be the polynomial ring in seven variables.
Consider the standard surjective homomorphism of $\F_q$-algebras $$\pi: S\ra\F_{q}[V\oplus V^{*}]^{\GL_{2}(\F_{q})}$$
defined by $C_i\mapsto c_i, C_i^*\mapsto c_i^*, U_j\mapsto u_j$
for $i\in\{0,1\}$ and $j\in\{-1,0,1\}$. Let $\I$ be the ideal generated by the five elements of $S$ that correspond to the left-hand sides of the relations (\ref{T1}), (\ref{TT1}), (\ref{T00}), (\ref{T01}), and (\ref{T10}), respectively. Clearly, $\I$ is contained in $\ker(\pi)$. 

The main result of this note is the following. 

\begin{thm}\label{stm}
$\ker(\pi)=\I$.
\end{thm}

We define $N$ to be the polynomial algebra generated by $\{c_{0},c_{1},c_{0}^{*},c_{1}^{*}\}$ over $\F_q$ and observe that $N$ is  a Noether normalization for $\F_{q}[V\oplus V^{*}]^{\GL_{2}(\F_{q})}$. It follows from \cite[Theorem 2.4]{BK11} and \cite[Theorem 1]{CHP91} that $\F_{q}[V\oplus V^{*}]^{\GL_{2}(\F_{q})}$ is Cohen-Macaulay, thus it is a free $N$-module. Moreover,
$\F_{q}[V\oplus V^{*}]^{\GL_{2}(\F_{q})}$ has a  basis $A\cup B\cup C$ over $N$, where
\begin{eqnarray*}
A&=&\{a(i,j,t):=  u_{-1}^{i}\cdot u_{1}^{j}\cdot (d_{2}^{*}\cdot d_{2})^{t}\mid  0\leqslant i,j\leqslant q-1, 0\leqslant t\leqslant q-2\},\\
 B&=& \{b(i,j,k,t):= u_{-1}^{i}\cdot u_{1}^{j}\cdot u_{0}^{k}\cdot (d_{2}^{*}\cdot d_{2})^{t}\mid 0\leqslant i,j\leqslant q-2, 1\leqslant k\leqslant q,0\leqslant t\leqslant q-2\},\\
 C&=&\{c(s,k,t):= (h_{s}\cdot d_{2}^{*s})\cdot u_{0}^{k}\cdot (d_{2}^{*}\cdot d_{2})^{t}\mid 1\leqslant s\leqslant q-2, 0\leqslant k\leqslant q-1,0\leqslant t\leqslant q-2\},
\end{eqnarray*}
see \cite[Theorem 3.2]{Che14}. We use $x(i,j,t), y(i,j,k,t)$ and $z(s,k,t)$ to denote the elements in $S$ that correspond to 
$a(i,j,t), b(i,j,k,t)$ and $c(s,k,t)$ respectively. For example, 
$$x(i,j,t)=U_{-1}^i\cdot U_1^j\cdot (U_{-1}\cdot U_{1}-U_{0}^{q+1})^t.$$
Also we use $\A,\B,\C$ to denote the subsets of $S$ corresponding to $A,B,C$ respectively. 
Let $\NN$ be the polynomial subalgebra of $S$ generated by $\{C_{0},C_{1},C_{0}^{*},C_{1}^{*}\}$ and 
$$\LL:=\sum_{\ell\in\A\cup\B\cup\C} \ell\cdot\NN.$$

\begin{proof}[Proof of Theorem \ref{stm}] To prove $\ker(\pi)\subseteq\I$, we see from \cite[Section 3.8.3]{DK15} that all algebraic relations among 
$\{c_{0},c_{1},c_{0}^{*},c_{1}^{*},u_{-1},u_{0},u_{1}\}$ can be generated by $N$-module relations among elements of $A\cup B\cup C$; and the latter consists of 
all linear and quadratic relations. Note that in our case, there are no linear relations  among elements of $A\cup B\cup C$ as $\F_{q}[V\oplus V^{*}]^{\GL_{2}(\F_{q})}$ is Cohen-Macaulay. Equivalently, this says, via pulling back in $S$, that
the set
$$K:=\{f\cdot g-\ell_{f,g}\in\ker(\pi)\mid f,g\in\A\cup\B\cup\C, \ell_{f,g}\in \LL\}$$
is a generating set for $\ker(\pi)$. Hence, it is sufficient to show 
that $K\subseteq \I$. On the other hand, Lemma \ref{keylem3} below shows that for any $f,g\in\A\cup\B\cup\C$, we have
$f\cdot g\in \LL$ modulo $\I$, i.e., there exists an $\ell_{f,g}\in \LL$ such that $f\cdot g-\ell_{f,g}\in \I\subseteq\ker(\pi).$ Therefore, $K\subseteq\I$, as desired. 
\end{proof}

\section{Several Lemmas} \label{sec5}
\setcounter{equation}{0}
\renewcommand{\theequation}
{5.\arabic{equation}}
\setcounter{theorem}{0}
\renewcommand{\thetheorem}
{5.\arabic{theorem}}

\noindent This sections consists of three technical lemmas that serve to the proof of Theorem \ref{stm}.
Throughout this section we are working over modulo $\I$.

\begin{lem}\label{keylem}
For any $\ell\in\LL$ and $n\in\N^+$, we have $U_0^n\cdot\ell\in \LL$ modulo $\I$.
\end{lem}

\begin{proof}
Since $U_0^n\cdot\ell=U_0^{n-1}\cdot(U_0\cdot\ell)$, it is sufficient to show that $U_0\cdot\ell\in \LL$. As any element of $\LL$ is an $\NN$-linear combination of  elements of $\A\cup\B\cup\C$, it suffices to show that
$U_0\cdot\ell\in\LL$ for all $\ell\in\A\cup\B\cup\C.$ Thus our arguments will be separated into three cases.

We first assume that $\ell=x(i,j,t)\in\A$. Here we have three subcases: 
(1) If $i,j\leqslant q-2$, then $U_0\cdot\ell=y(i,j,1,t)\in\B\subseteq\LL$.
(2) Assume that  $i=q-1$. We consider the case where $j\leqslant q-3$. By (\ref{Ks}) with $s=1$, we have $U_{-1}^{q-1}\cdot U_0=C_1^*\cdot U_1-z(1,0,0)$. 
Thus  $U_0\cdot\ell=U_0\cdot U_{-1}^{q-1}\cdot x(0,j,t)=(C_1^*\cdot U_1-z(1,0,0))\cdot x(0,j,t)=C_1^*\cdot x(0,j+1,t)-
z(1,0,0)\cdot x(0,j,t)$. Thus it suffices to show that $z(1,0,0)\cdot x(0,j,t)\in\LL$. It follows from (\ref{Rs}) with $s=1$ that
\begin{eqnarray*}
z(1,0,0)\cdot x(0,j,t)& = & (y(q-2,0,1,1)+z(2,0,0))\cdot x(0,j-1,t)\\
 & = & y(q-2,j-1,1,t+1)+z(2,0,0)\cdot x(0,j-1,t),
\end{eqnarray*}  
where $$y(q-2,j-1,1,t+1)=\begin{cases}
  y(q-2,j-1,1,t+1)\in\B, & t\leqslant q-3, \\
 C_0\cdot C_0^*\cdot  y(q-2,j-1,1,0)\in\LL, & t=q-2.
\end{cases}$$
Using (\ref{Rs}) $(s=2,\dots,q-3)$ and proceeding in the same way on 
$z(2,0,0)\cdot x(0,j-1,t)$, we see that there exists an $\ell'_j\in\LL$ such that
$z(1,0,0)\cdot x(0,j,t)=\ell'+z(j+1,0,t)\in\LL$ for all $j\leqslant q-3$. Hence, $U_0\cdot \ell\in\LL$ for $j\leqslant q-3$.
Further, we consider the cases where $j=q-2$ and $j=q-1$. Combining (\ref{T10}) and (\ref{K00}) implies that
$u_{-1}^{q-1}\cdot u_0\cdot u_{1}^{q-2}=c_{1}\cdot c_{0}^{*}-c_{1}^{*}\cdot u_{1}^{q-1}-u_{-1}\cdot u_{0}\cdot \delta,$
where $$\delta:=\sum_{i=2}^{q-1}\sum_{j=1}^{i-1} (-1)^{2i-1}{q-1\choose i}{i-1\choose j} (u_{-1}\cdot u_{1})^{q-j-2}\cdot (d_2\cdot d_2^*)^{j}.$$ Note that if $\Delta$ denotes the polynomial of $S$ corresponding to $\delta$, then
$U_{-1}\cdot U_{0}\cdot \Delta\cdot x(0,0,t)$  and $U_{-1}\cdot U_{0}\cdot \Delta\cdot U_1\cdot x(0,0,t)$ are both in $\LL.$ 
Thus
$U_{-1}^{q-1}\cdot U_0\cdot U_{1}^{q-2}\cdot x(0,0,t)=C_{1}\cdot C_{0}^{*}\cdot x(0,0,t)-C_{1}^{*}\cdot x(0,q-1,t)-U_{-1}\cdot U_{0}\cdot \Delta\cdot x(0,0,t)\in\LL$
and moreover,  it follows from (\ref{T1}) that 
$U_{-1}^{q-1}\cdot U_0\cdot U_{1}^{q-1}\cdot x(0,0,t)=C_{1}\cdot C_{0}^{*}\cdot x(0,1,t)-C_{1}^{*}\cdot(C_1\cdot y(0,0,q,t)-C_0\cdot x(1,0,t))-U_{-1}\cdot U_{0}\cdot \Delta\cdot U_1\cdot x(0,0,t)\in\LL.$ Therefore,
$U_0\cdot x(q-1,j,t)\in\LL.$
(3) Symmetrically, switching the roles of $U_{-1}$ and $U_1$ and using the $*$-images of the relations appeared in the previous case, one can shows that $U_0\cdot x(i,q-1,t)\in\LL.$

Secondly, we assume that $\ell=y(i,j,k,t)\in\B$. If $k\leqslant q-1$, then $U_0\cdot \ell= y(i,j,k+1,t)\in\B\subseteq \LL$.
If $k=q$, then $U_0\cdot \ell=U_0^{q+1}\cdot x(i,j,t)=(x(1,1,0)-x(0,0,1))\cdot x(i,j,t)$, by (\ref{K00}).
Further, since $i,j\leqslant q-2$, $x(1,1,0)\cdot x(i,j,t)=x(i+1,j+1,t)\in\A$. The fact that $c_0\cdot c_0^*=(d_2\cdot d_2^*)^{q-1}$ implies that 
$x(0,0,1)\cdot x(i,j,t)\in\LL.$
Hence, $U_0\cdot \ell$ is contained in $\LL$.

Thirdly, we assume that $\ell=z(s,k,t)\in\C$. If $k\leqslant q-2$, then $U_0\cdot \ell=z(s,k+1,t)\in\C\subseteq\LL$.
If $k=q-1$, then $U_0\cdot \ell=U_0^q\cdot z(s,0,t)$. The fact that
$u_0^q\cdot h_{s}\cdot d_{2}^{*s}=c_{0}^{*}\cdot u_{1}^{s+1}+u_{-1}^{q-s}\cdot(d_{2}\cdot d_{2}^{*})^{s}$ means that
\begin{eqnarray*}
U_0^q\cdot z(s,0,t) &=&U_0^q\cdot z(s,0,0)\cdot x(0,0,t) \\
&=& (C_{0}^{*}\cdot x(0,s+1,0)+x(q-s,0,s))\cdot x(0,0,t)\\
&=& C_{0}^{*}\cdot x(0,s+1,t)+x(q-s,0,s)\cdot x(0,0,t).
\end{eqnarray*}
As $s\leqslant q-2$, we see that $x(0,s+1,t)\in\A$. The fact that $c_0\cdot c_0^*=(d_2\cdot d_2^*)^{q-1}$ implies that $x(q-s,0,s)\cdot x(0,0,t)\in\LL.$ Hence, 
$U_0\cdot \ell=U_0^q\cdot z(s,0,t)\in \LL$.
\end{proof}

\begin{lem}\label{keylem2}
Let $x(i_1,j_1,t_1)$ and $x(i_2,j_2,t_2)$ be two any elements of $\A$. Then $x(i_1,j_1,t_1)\cdot x(i_2,j_2,t_2)\in \LL$ modulo $\I$.
\end{lem}

\begin{proof}
We may assume that $i_1+i_2=q\cdot\ell_i+i_3$, $j_1+j_2=q\cdot\ell_j+j_3$ and $t_1+t_2=(q-1)\cdot\ell_t+t_3$, where
$\ell_i,\ell_j,\ell_t\in\{0,1\}$. Note that if $\ell_i=0$, then $i_3\leqslant q-1$; and if $\ell_i=1$, then $i_3\leqslant q-2$.
Similarly, if $\ell_j=0$, then $j_3\leqslant q-1$; and if $\ell_j=1$, then $j_3\leqslant q-2$. If $\ell_t=0$, then $t_3\leqslant q-2$; and if $\ell_t=1$, then $t_3\leqslant q-3$.
It follows from (\ref{T00}), (\ref{T1}) and (\ref{TT1}) that
\begin{eqnarray*}
 x(i_1,j_1,t_1)\cdot x(i_2,j_2,t_2)&=&U_{-1}^{i_1+i_2}\cdot U_{1}^{j_1+j_2}\cdot (U_{-1}\cdot U_{1}-U_{0}^{q+1})^{t_1+t_2}  \\
 & = &  U_{-1}^{q\cdot\ell_i+i_3}\cdot U_{1}^{q\cdot\ell_j+j_3}\cdot (U_{-1}\cdot U_{1}-U_{0}^{q+1})^{(q-1)\cdot\ell_t+t_3}\\
 &=& (C_{1}^{*}\cdot U_{0}^{q}-C_{0}^{*}\cdot U_{1})^{\ell_i}(C_{1}\cdot U_{0}^{q}-C_{0}\cdot U_{-1})^{\ell_j}(C_0\cdot C_0^*)^{\ell_t}\cdot x(i_3,j_3,t_3).
\end{eqnarray*}
To complete the proof, we need to show that the right-hand side of the previous equality, denoted by $\Delta$, belongs to $\LL.$
Consider the following four cases: (1) If $(\ell_i,\ell_j)=(0,0)$, then $\Delta=(C_0\cdot C_0^*)^{\ell_t}\cdot x(i_3,j_3,t_3)\in \LL$.
(2) Suppose that $(\ell_i,\ell_j)=(1,0)$. By Lemma \ref{keylem}, we see that $U_{0}^{q}\cdot x(i_3,j_3,t_3)\in\LL$. Moreover, using 
(\ref{T1}), we have
$$U_1\cdot x(i_3,j_3,t_3)=\begin{cases}
 C_1 \cdot y(i_3,0,q,t_3)-C_0\cdot x(i_3+1,0,t_3)\in\LL,    & j_3=q-1, \\
  x(i_3,j_3+1,t_3)\in\A,    & j_3\leqslant q-2.
\end{cases}$$ 
Hence, $\Delta=(C_0\cdot C_0^*)^{\ell_t}\cdot x(i_3,j_3,t_3)\cdot (C_{1}^{*}\cdot U_{0}^{q}-C_{0}^{*}\cdot U_{1})\in \LL$. (3) Symmetrically, the statement also follows for the case when $(\ell_i,\ell_j)=(0,1)$. 
(4) Suppose $(\ell_i,\ell_j)=(1,1)$. Then $\Delta= (C_{1}^{*}\cdot U_{0}^{q}-C_{0}^{*}\cdot U_{1})(C_{1}\cdot U_{0}^{q}-C_{0}\cdot U_{-1})(C_0\cdot C_0^*)^{\ell_t}\cdot x(i_3,j_3,t_3)$.
Lemma \ref{keylem} implies that $U_0^{2q}\cdot x(i_3,j_3,t_3)\in\LL.$ Further, 
$U_0^q\cdot U_{-1}\cdot x(i_3,j_3,t_3)=U_0^q\cdot x(i_3+1,j_3,t_3)\in\LL, U_{1}\cdot U_0^q\cdot x(i_3,j_3,t_3)=U_0^q\cdot x(i_3,j_3+1,t_3)\in\LL$, and $U_{-1}\cdot U_1\cdot x(i_3,j_3,t_3)=x(i_3+1,j_3+1,t_3)\in\A\subseteq \LL$. Hence,
$\Delta\in\LL.$
The proof is completed.
\end{proof}

\begin{lem}\label{keylem3}
Let $f$ and $g$ be two any elements of $\A\cup\B\cup\C$. Then $f\cdot g\in \LL$ modulo $\I$.
\end{lem}

\begin{proof}
We have seen in Lemma \ref{keylem2} that if $f,g\in\A$, then the statement follows. We will use this fact and 
Lemma \ref{keylem} to prove the remaining cases. (1) Assume that $f=x(i_1,j_1,t_1)\in\A$ and $g= y(i_2,j_2,k,t_2)\in\B$. By Lemma \ref{keylem2}, we see that $x(i_1,j_1,t_1)\cdot x(i_2,j_2,t_2)\in\LL$. Thus
it follows from Lemma \ref{keylem} that $f\cdot g=U_0^k\cdot x(i_1,j_1,t_1)\cdot x(i_2,j_2,t_2)\in\LL$.
(2) Assume that $f=x(i,j,t)\in\A$ and $g=z(s,k,t')\in\C$.
It follows from (\ref{Ks}) that 
\begin{eqnarray*}
z(s,k,t')&=&U_0^k\cdot\left(C_{1}^{*}\cdot U_{1}^{s}+U_{-1}^{q-s}\cdot U_{0}\cdot\sum_{i=1}^s (-1)^i{s\choose i} (U_{-1}\cdot U_{1})^{s-i}\cdot U_{0}^{(q+1)(i-1)}\right) \cdot x(0,0,t')\\
&=&C_{1}^{*}\cdot U_0^k\cdot x(0,s,t')+ \sum_{i=1}^s (-1)^i{s\choose i} U_0^{(q+1)(i-1)+(k+1)} \cdot x(q-i,s-i,t').
\end{eqnarray*}
As $x(0,s,t')\in\A$ and $x(q-i,s-i,t')\in\A$ for all $1\leqslant i\leqslant s$, it follows from Lemma \ref{keylem2} that
$x(i,j,t)\cdot x(0,s,t')\in\LL$ and $x(i,j,t)\cdot x(q-i,s-i,t')\in\LL$ for all $1\leqslant i\leqslant s$. Hence, by Lemma \ref{keylem},
we see that $f\cdot g=x(i,j,t)\cdot z(s,k,t')\in \LL$. (3) For two elements $f=y(i_1,j_1,k_1,t_1)$ and $g=y(i_2,j_2,k_2,t_2)$ in $\B$, we see that $f\cdot g=U_0^{k_1+k_2}\cdot 
x(i_1,j_1,t_1)\cdot x(i_2,j_2,t_2)\in\LL$, by applying Lemmas \ref{keylem2} and \ref{keylem}. (4) 
Assume that $f=y(i,j,k,t)\in\B$ and $g=z(s,k',t')\in\C$. 
Note that $y(i,j,k,t)=U_0^k\cdot x(i,j,t)$ with $x(i,j,t)\in\A$, and the previous second case shows that
$x(i,j,t)\cdot z(s,k',t')\in \LL$. Now Lemma \ref{keylem} applies.
(5) Assume that $f=z(s_1,k_1,t_1)$ and $g=z(s_2,k_2,t_2)$ both are in $\C$. As in the second case, we expand $f$ as follows:
$$
f=C_{1}^{*}\cdot U_0^{k_1}\cdot x(0,s_1,t_1)+ \sum_{i=1}^{s_1} (-1)^i{s_1\choose i} U_0^{(q+1)(i-1)+(k_1+1)} \cdot x(q-i,s_1-i,t_1)
$$
where $x(0,s_1,t_1)$ and each $x(q-i,s_1-i,t_1)$ are in $\A$. By the above second case and Lemma \ref{keylem}, we see that 
$f\cdot g\in\LL$.
\end{proof}


\raggedright
\end{document}